\documentclass[12pt, reqno]{amsart}
\usepackage{amsmath}
\usepackage{amsfonts}
\usepackage{amssymb}
\usepackage{version}
\usepackage[all]{xy}
\begin{document}
\numberwithin{equation}{section}

\def\1#1{\overline{#1}}
\def\2#1{\widetilde{#1}}
\def\3#1{\widehat{#1}}
\def\4#1{\mathbb{#1}}
\def\5#1{\frak{#1}}
\def\6#1{{\mathcal{#1}}}

\newcommand{\de}{\partial}
\newcommand{\R}{\mathbb R}
\newcommand{\Ha}{\mathbb H}
\newcommand{\al}{\alpha}
\newcommand{\tr}{\widetilde{\rho}}
\newcommand{\tz}{\widetilde{\zeta}}
\newcommand{\tk}{\widetilde{C}}
\newcommand{\tv}{\widetilde{\varphi}}
\newcommand{\hv}{\hat{\varphi}}
\newcommand{\tu}{\tilde{u}}
\newcommand{\tF}{\tilde{F}}
\newcommand{\debar}{\overline{\de}}
\newcommand{\Z}{\mathbb Z}
\newcommand{\C}{\mathbb C}
\newcommand{\Po}{\mathbb P}
\newcommand{\zbar}{\overline{z}}
\newcommand{\G}{\mathcal{G}}
\newcommand{\So}{\mathcal{S}}
\newcommand{\Ko}{\mathcal{K}}
\newcommand{\U}{\mathcal{U}}
\newcommand{\B}{\mathbb B}
\newcommand{\oB}{\overline{\mathbb B}}
\newcommand{\Cur}{\mathcal D}
\newcommand{\Dis}{\mathcal Dis}
\newcommand{\Levi}{\mathcal L}
\newcommand{\SP}{\mathcal SP}
\newcommand{\Sp}{\mathcal Q}
\newcommand{\A}{\mathcal O^{k+\alpha}(\overline{\mathbb D},\C^n)}
\newcommand{\CA}{\mathcal C^{k+\alpha}(\de{\mathbb D},\C^n)}
\newcommand{\Ma}{\mathcal M}
\newcommand{\Ac}{\mathcal O^{k+\alpha}(\overline{\mathbb D},\C^{n}\times\C^{n-1})}
\newcommand{\Acc}{\mathcal O^{k-1+\alpha}(\overline{\mathbb D},\C)}
\newcommand{\Acr}{\mathcal O^{k+\alpha}(\overline{\mathbb D},\R^{n})}
\newcommand{\Co}{\mathcal C}
\newcommand{\Hol}{{\rm Hol}}
\newcommand{\Aut}{{\sf Aut}(\mathbb D)}
\newcommand{\D}{\mathbb D}
\newcommand{\oD}{\overline{\mathbb D}}
\newcommand{\oX}{\overline{X}}
\newcommand{\loc}{L^1_{\rm{loc}}}
\newcommand{\la}{\langle}
\newcommand{\ra}{\rangle}
\newcommand{\thh}{\tilde{h}}
\newcommand{\N}{\mathbb N}
\newcommand{\kd}{\kappa_D}
\newcommand{\Hr}{\mathbb H}
\newcommand{\ps}{{\sf Psh}}
\newcommand{\Hess}{{\sf Hess}}
\newcommand{\subh}{{\sf subh}}
\newcommand{\harm}{{\sf harm}}
\newcommand{\ph}{{\sf Ph}}
\newcommand{\tl}{\tilde{\lambda}}
\newcommand{\gdot}{\stackrel{\cdot}{g}}
\newcommand{\gddot}{\stackrel{\cdot\cdot}{g}}
\newcommand{\fdot}{\stackrel{\cdot}{f}}
\newcommand{\fddot}{\stackrel{\cdot\cdot}{f}}

\def\v{\varphi}
\def\Re{{\rm Re}\,}
\def\Im{{\rm Im}\,}
\def\ext{{\rm ext}\,}
\def\Lr{{\rm Lr}\,}
\def\tr{{\rm tr}\,}
\def\sp{{\rm sp}}
\def\rg{{\rm rg}\,}

\def\Label#1{\label{#1}}


\def\cn{{\C^n}}
\def\cnn{{\C^{n'}}}
\def\ocn{\2{\C^n}}
\def\ocnn{\2{\C^{n'}}}
\def\je{{\6J}}
\def\jep{{\6J}_{p,p'}}
\def\th{\tilde{h}}


\def\dist{{\rm dist}}
\def\const{{\rm const}}
\def\Span{{\rm Span\,}}
\def\rk{{\rm rank\,}}
\def\dim{{\rm dim}}
\def\id{{\rm id}}
\def\aut{{\rm aut}}
\def\hol{{\rm hol}}
\def\Aut{{\rm Aut}}
\def\CR{{\rm CR}}
\def\GL{{\sf GL}}
\def\U{{\sf U}}
\def\la{\langle}
\def\ra{\rangle}

\emergencystretch15pt \frenchspacing

\newtheorem{theorem}{Theorem}[section]
\newtheorem*{theorem**}{Theorem \mynumber}
\newenvironment{theorem*}[1]
  {\newcommand{\mynumber}{\ref{#1}}\begin{theorem**}}
  {\end{theorem**}}
\newtheorem{lemma}[theorem]{Lemma}
\newtheorem{proposition}[theorem]{Proposition}
\newtheorem{problem}[theorem]{Problem}
\newtheorem{corollary}[theorem]{Corollary}

\theoremstyle{definition}
\newtheorem{definition}[theorem]{Definition}
\newtheorem{example}[theorem]{Example}

\theoremstyle{remark}
\newtheorem{remark}[theorem]{Remark}
\numberwithin{equation}{section}

\title[Loewner equations on complete hyperbolic domains]{Loewner equations on complete hyperbolic domains}
\author[L. Arosio]{Leandro Arosio$^{\ddagger}$}
\address{Istituto Nazionale di Alta Matematica ``Francesco Severi'', Citt\`a Universitaria, Piazzale Aldo Moro 5, 00185 Rome, Italy}
\email{arosio@altamatematica.it}
\thanks{$^{\ddagger}$Titolare di una Borsa della Fondazione Roma - Terzo Settore  bandita dall'Istituto Nazionale di Alta Matematica}

\date{\today }


\setcounter{tocdepth}{1}

\begin{abstract}
We prove that,  on a complete hyperbolic domain $D\subset \mathbb{C}^q$,  any Loewner PDE associated with a Herglotz vector field of the form $H(z,t)=\Lambda(z)+O(|z|^2)$, where   the eigenvalues of $\Lambda$ have strictly negative real part, admits a solution given by a family of univalent mappings $(f_t\colon D\to \mathbb{C}^q)$ which satisfies $\cup_{t\geq 0}f_t(D)=\mathbb{C}^q$.   If no real resonance occurs among the eigenvalues of $\Lambda$, then the family $(e^{\Lambda t}\circ f_t)$ is uniformly bounded in a neighborhood of the origin.  We also give a generalization of Pommerenke's univalence criterion on complete hyperbolic domains.
\end{abstract}
\maketitle
\tableofcontents
\section{Introduction}

We begin recalling the Loewner equations on the unit disc $\mathbb{D}\subset \mathbb{C}.$
The  Loewner PDE is the following:
\begin{equation}\label{L1}
\frac{\partial f_t(z)}{\partial t}=-\frac{\partial f_t(z)}{\partial z} H(z,t),\quad \hbox{a.e. } t\geq 0,z\in \mathbb{D},
\end{equation}
where $H(z,t)=zp(z,t)$ and $p(z,t)\colon \mathbb{D}\times \mathbb{R}^+\to \C$ is measurable in $t\geq 0$,   holomorphic in $z\in \mathbb{D}$ and satisfies $ \Re p(z,t)<0$ and $p(0,t)=-1$ for all $t\geq 0$.  The second equation is the Loewner ODE:
\begin{equation}
\begin{cases}
\frac{\de}{\de t}\v_{s,t}(z)= H(\v_{s,t}(z),t), \quad \hbox{a.e. } t\in [s,\infty), z\in \mathbb{D},\\
\v_{s,s}(z)=z,\quad s\geq 0,z\in \mathbb{D}.
\end{cases}
\end{equation}
Both equations were introduced by  Loewner in 1923 \cite{Loewner} and used to prove the case $n=3$ of the Bieberbach conjecture. Loewner theory was developed by Pommerenke \cite{Pommerenke} and  Kufarev \cite{Kufarev} as a powerful tool in geometric function theory. In fact it is  one of the main ingredients of the proof of the  Bieberbach conjecture given by de Branges \cite{deBranges}  in 1985.
Among the extensions of  the  theory we recall the celebrated theory of Schramm-Loewner evolution \cite{Schramm} introduced in 1999.

Loewner thoery  was extended to  several complex variables  by Duren, Graham, Hamada, G. Kohr, M. Kohr,  Pfaltzgraff  and others \cite{Duren-Graham-Hamada-Kohr}\cite{Graham-Hamada-Kohr-Kohr}\cite{Pfaltzgraff}.
Recently Bracci, Contreras and D\`iaz-Madrigal \cite{Bracci-Contreras-Diaz}\cite{Bracci-Contreras-Diaz-II} (see also \cite{Arosio-Bracci}) proposed a generalization of the  Loewner ODE which has its natural setting in  complete hyperbolic manifolds.
In the following we denote by $D$ a complete hyperbolic (in the sense of Kobayashi) domain of $\C^q$.
Recall that a holomorphic vector field $H\colon D\to \C^q$ is said an {\sl infinitesimal generator} provided the Cauchy problem
$$
\begin{cases}
 \overset{\bullet}{z}(t)=H(z(t)),\\
z(0)=z_0
\end{cases}
$$
has a solution $z:[0,+\infty)\to D$ for all $z_0\in D$.

 A {\sl Herglotz vector field} on a complete hyperbolic domain  $D\subset \C^q$ is  a non-autonomous holomorphic vector field $H(z,t)\colon D\times\R^+\to \C^q$ which is measurable in $t\geq 0$, which is an infinitesimal generator for a.e. $t\geq 0$ fixed, and such that
for any compact set $K\subset D$   there exists a  function $c_{K}\in L^d_{loc}(\R^+,\R^+)$, with $d\in [1,\infty]$, such that $$|H(z,t)|\leq c_{K}(t),\quad z\in K, t\geq 0.$$ These vector fields are the natural generalizations of the function $H(z,t)=-zp(z,t)$ in (\ref{L1}). The Loewner ODE studied in \cite{Arosio-Bracci}\cite{Bracci-Contreras-Diaz-II}
\begin{equation}
\label{ODE}
\begin{cases}
\frac{\de}{\de t}\v_{s,t}(z)= H(\v_{s,t}(z), t), \quad \hbox{a.e. } t\in[s,\infty),\ z\in D,\\
\v_{s,s}(z)=z,\quad s\geq 0,\ z\in D,
\end{cases}
\end{equation}
has a locally absolutely continuous (in the variable $t$) solution defined for all $0\leq s\leq t$ given by a family $(\v_{s,t}\colon D\to D)$ of univalent mappings which is a $\R^+$-evolution family, that is which satisfies $\v_{s,t}=\v_{u,t}\circ\v_{s,u}$ for all $0\leq s\leq u\leq t$ and $\v_{s,s}(z)=z$ for all $s\geq 0$.

Let $H(z,t)$ be a Hergloz vector field on $D$.
In \cite{Arosio-Bracci-Hamada-Kohr} we proved that a  family of univalent mappings  $(f_t\colon D\to \C^q)$ is locally absolutely continuous (in the variable $t$)  and solves the Loewner PDE 
\begin{equation}\label{PDEintro}
\frac{\partial f_t(z)} {\partial t }=-df_t(z)H(z,t),\quad \hbox{a.e. } t\geq 0,\  z\in D,
\end{equation} if and only if it solves the functional equation 
\begin{equation}\label{functional}
f_t\circ\v_{s,t}(z)=f_s(z),\quad 0\leq s\leq t,\ z\in D,
\end{equation}
where $(\v_{s,t})$ is the solution of (\ref{ODE}).

The solution $(f_t\colon D\to \C^q)$ satisfies $f_s(D)\subset f_t(D)$ for all $0\leq s\leq t$. A family of univalent mappings with this property is called  a $\mathbb{R}^+$-{\sl Loewner chain}.
  A $\mathbb{R}^+$-evolution family $(\v_{s,t})$ and a $\mathbb{R}^+$-Loewner chain $(f_t)$ are {\sl associated} if (\ref{functional}) holds.

We now  introduce the special Herglotz vector fields that we are going to study in this paper.
A Herglotz vector field $H(z,t)$ is of {\sl dilation type} if  $$H(z,t)=\Lambda(z)+O(|z|^2), \quad t\geq 0,$$ where the eigenvalues of $\Lambda\in \mathcal{L}(\mathbb{C}^q)$ have strictly negative real part, and the term $O(|z|^2)$ may depend on $t$.

We recall the following recent result by Graham, Hamada, G. Kohr and M. Kohr \cite{Graham-Hamada-Kohr-Kohr}.
\begin{theorem}[\cite{Graham-Hamada-Kohr-Kohr}]\label{kohrresult}
Let $H(z,t)=\Lambda(z)+O(|z|^2)$ be a dilation Herglotz vector field on the unit ball $\mathbb{B}\subset \C^q$, and assume that
\begin{equation}\label{condition}
2\max\{\Re\langle \Lambda(z),z\rangle:|z|=1\}<\min \{ \Re\lambda :\lambda \in \sp(\Lambda)\}.
\end{equation}
Then the Loewner PDE (\ref{PDEintro}) admits a locally absolutely continuous  univalent solution $(f_t\colon \mathbb{B}\to \C^q)$ such that  $\cup_{t\geq 0}f_t(\mathbb{B})=\mathbb{C}^q$. The family $(e^{\Lambda t}\circ f_t)$ is uniformly bounded in a neighborhood of the origin.
\end{theorem}

In \cite{Arosio} we introduced  $\mathbb{N}$-evolution families and $\mathbb{N}$-Loewner chains, that is the discrete-time analogues  of  $\mathbb{R}^+$-evolution families and  $\mathbb{R}^+$-Loewner chains. Solving equation (\ref{functional}) for discrete times we proved the following result.
\begin{theorem}[\cite{Arosio}]\label{continuousresultintro}
Let $H(z,t)=\Lambda(z)+O(|z|^2)$ be a dilation Herglotz vector field on the unit ball $\mathbb{B}\subset \C^q$. Assume that $\Lambda$ is diagonal. Then the Loewner PDE (\ref{PDEintro}) admits a locally absolutely continuous univalent solution $(f_t\colon \mathbb{B}\to \C^q)$ such that  $\cup_{t\geq 0}f_t(\mathbb{B})=\mathbb{C}^q$. The family $(e^{\Lambda t}\circ f_t)$ is uniformly bounded in a neighborhood of the origin if no real resonance of the form
 $$\Re(\sum_{j=1}^q k_j\alpha_j)=\Re\alpha_l,\quad k_j\geq 0,\ \sum_{j=1}^q k_j\geq 2$$ occurs among the eigenvalues $(\alpha_j)$ of $\Lambda$. 
\end{theorem}
The same result was obtained independently  with  different methods by Voda \cite{Voda}, assuming $\max\{\Re\langle \Lambda(z),z\rangle:|z|=1\}<0$  instead of  assuming $\Lambda$ diagonal.

Notice that condition (\ref{condition})  avoids real resonances. 
Recall that by \cite[Counterexemple 2]{Arosio} the Loewner PDE associated with the autonomous dilation Herglotz vector field on $\B\subset \C^2$ $$H(z,t)=(\alpha z_1,2\alpha z_2 +cz_1^2),$$ where $|\alpha|<1/2$ and $c\in \C^*$ is small enough, does not admit any solution $(f_t\colon \mathbb{B}\to \C^q)$ such that  the family $(e^{\Lambda t}\circ f_t)$ is uniformly bounded in a neighborhood of the origin. In this case a real resonance occurs.

In this paper we generalize Theorem \ref{continuousresultintro} to any  dilation  Herglotz vector field on a complete hyperbolic domain  $D\subset\mathbb{C}^q$. We should mention that, to our knowledge, this is the first existence result for the Loewner PDE (\ref{PDEintro}) on such  domains. 
 We start by solving equation (\ref{functional})  for discrete times. Let $(\v_{n,m})$ be a $\mathbb{N}$-evolution family. We show that a family  of tangent to identity univalent  mappings $(h_n\colon D\to \C^q)$ 
which is uniformly bounded near the origin  solves the {\sl non-autonomous Schr\"oder equation}
\begin{equation}\label{lastbutnotleast}
h_m\circ \v_{n,m}=e^{\Lambda (m-n)}\circ h_n.
\end{equation}
 if and only if $(\v_{n,m})$ is associated with the  $\mathbb{N}$-Loewner chain $(f_n)\doteq(e^{-\Lambda n}\circ h_n).$ 

Equation  (\ref{lastbutnotleast})   shows a strong connection between Loewner theory and the theory of basins of attraction of discrete non-autonomous complex dynamical systems grown around  Bedford's conjecture: see \cite{Abate-Abbondandolo-Majer}\cite{Fornaess-Stensones}\cite{Jonsson-Varolin}\cite{Peters}\cite{Wold}\cite{Guitta}.
To solve equation (\ref{lastbutnotleast}) we use techniques from this theory, in particular from \cite{Peters}. Indeed we need a non-autonomous version of the Poincar\'e-Dulac method, whose homological equation is replaced  by a difference equation in the space $\mathcal{H}_i$ of homogeneus polynomial mappings of degree $i$,
\begin{equation}\label{ricorrenza}
H_{n+1}=e^{\Lambda}\circ H_n\circ e^{-\Lambda}+B_n,
\end{equation} 
where $(H_n)$ is an unknown bounded sequence in $\mathcal{H}_i$ and $(B_n)$ is a bounded sequence in $\mathcal{H}_i$. In order to find a bounded solution of (\ref{ricorrenza}) we study the spectral and dynamical properties of the linear operator  $H\mapsto e^{\Lambda}\circ H\circ e^{-\Lambda}$ acting on $\mathcal{H}_i$ and we show that the  obstruction to the existence of solutions  is given by  real resonances.

This method provides a family of univalent mappings $(f_n\colon r\mathbb{B}\subset D\to \C^q)_{n\in \mathbb{N}}$ satisfying (\ref{functional}) but defined only for integer times and in a little neighborhood of the origin. Then we extend this family to all $t\in\mathbb{R}^+$ and $z\in D$.

The main result of this paper is thus the following.
\begin{theorem}\label{convexresultintro}
Let $D\subset \mathbb{C}^q$ be a complete hyperbolic domain and let $H(z,t)=\Lambda(z)+O(|z|^2)$ be a dilation Herglotz vector field on D. Then the Loewner PDE (\ref{PDEintro}) admits a locally absolutely continuous univalent solution $(f_t\colon \mathbb{B}\to \C^q)$ such that  $\cup_{t\geq 0}f_t(D)=\mathbb{C}^q$.  The family $(e^{\Lambda t}\circ f_t)$ is uniformly bounded in a neighborhood of the origin if no real resonance occurs among the eigenvalues of $\Lambda$.

\end{theorem}

We also generalize to complete hyperbolic domains the  classical univalence criterion in the unit disk due to Pommerenke \cite[Folgerung 6]{Pommerenke}. 

\begin{theorem}\label{univalencecriterionintro}
Let $D\subset \mathbb{C}^q$ be a complete hyperbolic domain and let $H(z,t)=\Lambda(z)+O(|z|^2)$ be a dilation Herglotz vector field on D.  
Let  $(f_t\colon D\to \C^q)$ be a family of holomorphic mappings which solves the Loewner PDE (\ref{PDEintro}) and assume that the family $(e^{\Lambda t}\circ f_t)$ is uniformly bounded in a neighborhood of the origin is an univalent family. Then for all $t\geq 0$ the mapping $f_t$ is univalent.
\end{theorem}

\section{Local conjugacy}\label{locconj}

We start recalling some basic definitions.

\begin{definition} Let $\mathbb{T}$ be $\mathbb{N}$ or $\mathbb{R}^+$.
Let $D$ be a domain of $\C^q$. A {\sl $\mathbb{T}$-evolution family} is a family of univalent mappings $(\v_{\alpha,\beta}\colon D\to D)_{\alpha\leq \beta\in \mathbb{T}}$  such that 
\begin{itemize}
\item[i)]$ \v_{\al,\al}=\id$ for all $\al\geq 0$,
\item[ii)]  $\v_{\al,\beta}=\v_{\gamma,\beta}\circ\v_{\al,\gamma}$ for all $0\leq \al\leq \gamma\leq \beta.$
\end{itemize}

A family  of univalent mappings $(f_\al\colon D\to \C^q)$ is a {\sl $\mathbb{T}$-Loewner chain} if $f_\al(D)\subset f_\beta(D)$  for all $0\leq\al\leq \beta$. 
\end{definition}
\begin{remark}
Let  $(f_\al\colon D\to \C^q)$ be a $\mathbb{T}$-Loewner chain. Then there exists a unique associated  $\mathbb{T}$-evolution family $(\v_{\al,\beta}\doteq f^{-1}_\beta\circ f_\al)$.
\end{remark}

One has the following uniqueness result for $\mathbb{T}$-Loewner chains.
\begin{theorem}[\cite{Arosio-Bracci-Hamada-Kohr}]\label{thuniqueness}
Let $(\v_{\al,\beta})$ be a $\mathbb{T}$-evolution family on $D$ and let $(f_\al\colon D\to \C^q)$ be an associated  $\mathbb{T}$-Loewner chain
 If $(g_\al\colon D\to \C^q)$ is a subordination chain associated with  $(\v_{\al,\beta})$  then there exists a holomorphic mapping $\Psi\colon \cup_{\al\in \mathbb{T}}f_\al(D) \to \C^q$ such that $(g_\al=\Psi\circ f_\al).$ 
\end{theorem}

In what follows we focus on special types  of $\mathbb{N}$-evolution families and $\mathbb{N}$-Loewner chains.

\begin{definition}
We denote  $\mathcal{L}(\mathbb{C}^q)$ and $\mathcal{A}(\mathbb{C}^q)$  the sets of $\mathbb{C}$-linear endomorphisms and $\mathbb{C}$-linear automorphisms  of $\mathbb{C}^q$.
Let $A\in \mathcal{L}(\mathbb{C}^q)$. The {\sl spectrum} $\sigma(A)$ of $A$ is the set of its eigenvalues. The {\sl spectral radius} $\rho(A)$ is defined as $\max_{\lambda\in\sigma(A)}|\lambda|.$

Let $D$ be a  domain  in $\mathbb{C}^q$ containing $0$. 
A $\mathbb{N}$-evolution family $(\v_{n,m})$ on $D$   is a {\sl dilation $\mathbb{N}$-evolution family} if for all $  n\geq 0$,
\begin{equation}\label{discrete dilation type}
 \varphi_{n,n+1}(z)=A(z)+O(|z|^2),
\end{equation}
with $A\in \mathcal{A}(\mathbb{C}^q)$ such that $\rho(A)<1.$
A $\mathbb{N}$-Loewner chain $(f_n\colon D\to \mathbb{C}^q)$ is a {\sl locally bounded $\mathbb{N}$-Loewner chain} if for all $n\geq 0$, $$f_n(z)=A^{-n}(z)+O(|z|^2),$$ where
 $A\in \mathcal{A}(\mathbb{C}^q)$ is such that $\rho(A)<1$ and the family $(A^n\circ f_n)$ is uniformly bounded in a neighborhood of the origin.

On complete hyperbolic domains, the dynamics of dilation $\mathbb{N}$-evolution families is uniformly contractive, as the following lemma shows. A reference for complete hyperbolic manifolds and the Kobayshi distance is \cite{Kobayashi}.
\begin{lemma}\label{nuovo}
Let $D\subset \C^q$ be a complete hyperbolic domain and  let $(\varphi_{n,m})$   be dilation $\mathbb{N}$-evolution family  on $D$. Then  the basin of attraction of the origin at time $n\geq 0$
 $$\mathfrak{A}(n)\doteq\{z\in D: \lim_{m\to \infty} \varphi_{n,m}(z)=0\}$$ is the whole $D$,
and for all $n\geq 0$ the convergence $\lim_{m\to \infty} \varphi_{n,m}(z)=0$ is uniform on compact subsets.

\end{lemma}
\begin{proof}
Up to a linear change of coordinates, we may assume that $ \max_{z\in \mathbb{C}^q}\frac{|A(z)|}{|z|}<1$. Lemma \ref{estimate} yields then  that there exists  $\varepsilon >0$ such that the Kobayashi ball $\Omega(0,\varepsilon)$ centered in the origin of radius $\varepsilon$ is contained in the set $\bigcap_{m\geq 0}\mathfrak{A}(m)$. For all $n\geq 0$, the set  $\mathfrak{A}(n)$ is an open subset of $D$. Indeed,  if $z\in \mathfrak{A}(n)$, there exists $m>0$ such that $\varphi_{n,m}(z)\subset \Omega(0,\varepsilon/2)$. Since holomorphic mappings decrease the Kobayashi distance, one has $$\v_{n,m}(\Omega(z,\varepsilon/2 ))\subset \Omega(0,\varepsilon)\subset \mathfrak{A}(m),$$ thus $\Omega(z,\varepsilon/2 )\subset \mathfrak{A}(n)$.

The set $\mathfrak{A}(n)$ is also a closed subset of $D$. Indeed let $z$ be a point in the closure of  $\mathfrak{A}(n)$.
Then there exist a point $w\in\mathfrak{A}(n)$ such that $k_D(z,w)<\varepsilon/2$. Let $u>0$ be such that $\varphi_{n,u}(w)\in  \Omega(0,\varepsilon/2)$.  Since holomorphic mappings decrease  the Kobayashi distance  one has $$\varphi_{n,u}(z)\subset  \Omega(0,\varepsilon)\subset \mathfrak{A}(u),$$
thus $z\in \mathfrak{A}(n)$.

Since $D$ is connected one has $\mathfrak{A}(n)=D$. The convergence is local uniform and hence  uniform on compact subsets.

\end{proof}

A \textit{triangular mapping} is a mapping $T\colon\mathbb{C}^q\rightarrow\mathbb{C}^q$ whose components $T^{(i)}(z)$ satisfy
$$T^{(1)}(z)=\lambda_1z_1,\quad T^{(i)}(z)=\lambda_iz_i+t^{(i)}(z_1,z_2,\dots,z_{i-1}),\ 2\leq i\leq q,$$
where $\lambda_i\in\mathbb{C}$ and $t^{(i)}$ is a polynomial in $i-1$ variables fixing the origin. Its {\sl degree } is the maximum of the degree of its components.
If $\lambda_i\neq 0$ for all $1\leq i\leq q$, the mapping $T$ is called a {\sl triangular automorphism}.
This is indeed an automorphism of $\C^q$, since we can iteratively write its inverse, which is still a triangular automorphism.
Since the composition of two triangular automorphisms is still a triangular automorphism, they form a subgroup of $\aut(\mathbb{C}^q).$
A \textit{triangular dilation $\mathbb{N}$-evolution family} is a dilation $\mathbb{N}$-evolution family $(T_{n,m},\mathbb{C}^q)$ such that each $T_{n,n+1}$, and hence every $T_{n,m}$, is a triangular automorphism of $\mathbb{C}^q$. 
A triangular dilation $\mathbb{N}$-evolution family $(T_{n,m})$ has {\sl uniformly bounded coefficients} if the family $(T_{n,n+1})$ has uniformly bounded coefficients.  A  triangular dilation $\mathbb{N}$-evolution family $(T_{n,m})$ has {\sl uniformly bounded degree}  if the family $(T_{n,n+1})$ has uniformly bounded degree.
\end{definition}

\begin{definition}
Let $D\subset\C^q$ be a complete hyperbolic domain.
A dilation  $\mathbb{N}$-evolution family $(\v_{n,m}\colon D\to D )$ and  a  triangular dilation $\mathbb{N}$-evolution family $(T_{n,m})$ with uniformly bounded degree and uniformly bounded coefficients  are {\sl  locally conjugate} if there exists, on a  ball $r\mathbb{B}\subset D$ satisfying $$\v_{n,m}(r\mathbb{B})\subset(r\mathbb{B}),\quad 0\leq n\leq m,$$  a uniformly bounded family of  holomorphic mappings  $(h_n\colon r\mathbb{B}\to \mathbb{C}^q)$ such that $h_n(z)=z+O(|z|^2)$ for all  $n\geq 0,$ and such that 
\begin{equation}\label{coniugio}
h_m\circ\v_{n,m}(z)=T_{n,m}\circ h_n(z), \quad z\in r\mathbb{B},\ 0\leq n\leq m.
\end{equation}
\end{definition}

\begin{proposition}\label{inter}
Let $D\subset\C^q$ be a complete hyperbolic domain.
Assume that a dilation  $\mathbb{N}$-evolution family $(\v_{n,m}\colon D\to D)$ and  a  triangular dilation $\mathbb{N}$-evolution family $(T_{n,m})$ with uniformly bounded degree and uniformly bounded coefficients  are  locally conjugate by $(h_n\colon r\mathbb{B}\to \mathbb{C}^q)$.
Then for each  fixed $n\geq 0$ the sequence $(T_{n,m}^{-1}\circ h_m\circ \v_{n,m})_{m\geq n}$ is eventually defined on each compact subset $K\subset D$,  its limit $$h_n^{\sf e}\doteq\lim_{m\to\infty}T_{n,m}^{-1}\circ h_m\circ \v_{n,m}$$ exists uniformly on compacta on $D$, and satisfies  $h_n^{\sf e}|_{ r\mathbb{B}}=h_n$. The family $(h_n^{\sf e}\colon D\to \C^q)$   satisfies
$$h_m^{\sf e}\circ \v_{n,m}(z)=T_{n,m}\circ h_n^{\sf e}(z), \quad z\in D,\ 0\leq n\leq m.$$
\end{proposition}
\begin{proof}
 Let $K\subset D$ be a compact subset. By Lemma \ref{nuovo}, for all $n\geq 0$  there exists 
$u\geq n$ such that  $\v_{n,u}(K)\subset  r\mathbb{B}$ . Then  for $m\geq u$, 
$$ T_{n,m}^{-1}\circ h_m\circ \varphi_{n,m}|_K=T_{n,u}^{-1}\circ (T_{u,m}^{-1}\circ h_m\circ \v_{u,m})\circ \v_{n,u}|_K=T_{n,u}^{-1}\circ h_{u}\circ \v_{n,u}|_K$$ by (\ref{coniugio}), thus the sequence $(T_{n,m}^{-1}\circ h_m\circ \v_{n,m})_{m\geq n}$ converges uniformly on compacta. By  (\ref{coniugio}) we have, $$T_{n,m}^{-1}\circ h_m\circ \varphi_{n,m}(z)=h_n(z), \quad z\in r\mathbb{B},\ n\leq m,$$ thus $h_n^{\sf e}|_{r\mathbb{B}}=h_n$.

Finally
$$h_m^{\sf e}\circ \v_{n,m}=\lim_{j\to\infty}T_{m,j}^{-1}\circ h_j\circ \v_{m,j}\circ\v_{n,m}=T_{n,m}\circ \lim_{j\to\infty}T_{n,j}^{-1}\circ h_j\circ \v_{n,j}=T_{n,m}\circ h_n^{\sf e}.$$

\end{proof}
\begin{definition}
We call the mappings  $h_n^{\sf e}$    {\sl intertwining mappings}. Notice that since $h_n^{\sf e}|_{ r\mathbb{B}}=h_n$, the family $(h_n^{\sf e}\colon D\to \C^q)$ is uniformly bounded in a neighborhood of the origin.   From now on we will denote $h_n^{\sf e}$ simply by $h_n$.
\end{definition}
\begin{proposition}\label{interuniv}
Let $D\subset\C^q$ be a complete hyperbolic domain.
Assume that a dilation  $\mathbb{N}$-evolution family $(\v_{n,m}\colon D\to D)$ and  a  triangular dilation $\mathbb{N}$-evolution family $(T_{n,m})$ with uniformly bounded degree and uniformly bounded coefficients  are  locally conjugate.
Then each intertwining mapping $h_n\colon D\to \C^q$ is univalent.
\end{proposition}
\begin{proof}
Assume that there exist $z\neq w$ in $D$ and $n\geq 0$ such that $h_n(z)=h_n(w).$ Then by $(\ref{coniugio})$, 
\begin{equation}\label{blabla}
h_m(\v_{n,m}(z))=h_m(\v_{n,m}(w)), \quad 0\leq n\leq m.
\end{equation}
By Lemma \ref{koebe} there exists a  ball $s\mathbb{B}$ such that for all $m\geq 0$ the mapping $h_m|_{s\mathbb{B}}$ is univalent. By Lemma \ref{nuovo} there exists $m\geq n$ such that $\v_{n,m}(z)\cup \v_{n,m}(w)\subset s\mathbb{B}.$ But $\v_{n,m}(z)\neq \v_{n,m}(w)$ since $\v_{n,m}$ is a univalent mapping, hence (\ref{blabla})  contradicts the univalence of $h_m|_{s\mathbb{B}}$.
\end{proof}

\section{Non-autonomous Poincar\'e-Dulac method}\label{spectralsection}
For a detailed exposition of the classical  Poincar\'e-Dulac method, see \cite[Appendix]{Rudin-Rosay}. 
We will need the non-autonomous version of the Poincar\'e-Dulac method developed  in \cite{Peters}  in the case of $\mathbb{N}$-evolution families of holomorphic automorphisms of $\mathbb{C}^q$.  
We will give alternative proofs of this method in Propositions \ref{core} and \ref{discreteresult} in which we show also that, in absence of real resonances (defined below), it is possible to find a local conjugacy between a dilation $\mathbb{N}$-evolution family and its linear part. 

In what follows we identify a linear automorphism $A\in \mathcal{A}(\mathbb{C}^q)$ with its associated matrix with respect to the canonical basis.
\begin{definition}
A {\sl real multiplicative resonance} for $A\in\mathcal{A}(\mathbb{C}^q)$  with eigenvalues $\lambda_i$  is an identity $$|\lambda_j|=|\lambda_1^{i_1}\dots\lambda_q^{i_q}|,$$ where $i_j\geq 0$, and $\sum_j i_j\geq 2$. If for every $1\leq j\leq q$ we have $|\lambda_j|<1$, real multiplicative resonances can occur only in a finite number. Moreover, if $0<|\lambda_q|\leq\cdots\leq |\lambda_1|<1$, then 
\begin{equation}\label{risontriang}|\lambda_j|=|\lambda_1^{i_1}\dots\lambda_q^{i_q}|\Rightarrow i_j=i_{j+1}=\dots=i_q=0.
\end{equation}
\end{definition}
\begin{definition}\label{spectral}
An automorphism $A\in \mathcal{A}(\mathbb{C}^q)$ is in {\sl optimal form} if
\begin{enumerate}
\item[i)] $A$ is  in lower-triangular $\varepsilon$-Jordan normal form for some $\varepsilon>0$, that is in lower triangular Jordan normal form with the underdiagonal multiplied by $\varepsilon$,
\item[ii)] if the diagonal of  $A$ is $(\lambda_1,\dots,\lambda_q)$ then  $1>|\lambda_1|\geq\dots\geq |\lambda_q|>0,$
\item[iii)] one has $\max_{z\in \mathbb{C}^q}\frac{|A(z)|}{|z|}<1$.
\end{enumerate}
Note that any linear automorphism can be put in optimal form by a linear change of coordinates. 

Let $A\in \mathcal{A}(\mathbb{C}^q)$ be in optimal form. For $1\leq j\leq q$ let $\pi_j\colon \C^q\to \C$ be the projection to the $j$-th coordinate.
Let $i\geq 2$ and let $\mathcal{H}_i$ be the vector space of all holomorphic maps $H:\mathbb{C}^q\rightarrow \mathbb{C}^q$ whose components $\pi_j\circ H$ are homogeneus polynomials of degree $i$. 
A basis  for this vector space is easily described: let $1\leq j\leq q$,  let $I\in \mathbb{N}^q$ be a multi-index of absolute value $|I|=i,$ and define $X^j_I$ such that $$\pi_l\circ X^j_I\doteq \delta_{l,j}z^I,\quad 1\leq l\leq q.$$ 
The set $\mathfrak{B}\doteq\{ X^j_I:\ 1\leq j\leq q,|I|=q\}$ is a basis of  $\mathcal{H}_i$.
Next we define a splitting of  $\mathcal{H}_i$ by specifying a partition of  the basis $\mathfrak{B}$.

We set  $X^j_I\in \mathfrak{B}_r$ if $|\lambda_j\lambda^{-I}|=1$.
The {\sl real resonant subspace} $\mathcal{R}_i$ is the vector subspace spanned by the vectors in  $\mathfrak{B}_r$.

We set  $X^j_I\in \mathfrak{B}_s$ if $|\lambda_j\lambda^{-I}|<1$.
The {\sl stable subspace} $\mathcal{S}_i$ is the vector subspace spanned by the vectors in  $\mathfrak{B}_s$.

We set  $X^j_I\in \mathfrak{B}_u$ if $|\lambda_j\lambda^{-I}|>1$.
The {\sl unstable subspace} $\mathcal{U}_i$ is the vector subspace spanned by the vectors in  $\mathfrak{B}_u$.

This defines the splitting  $\mathcal{H}_i=\mathcal{R}_i\oplus\mathcal{S}_i\oplus\mathcal{U}_i,$ with  projections $\pi_r,$ $\pi_s,$ and $\pi_u$.

If $F\in \mathcal{L}(\mathbb{C}^q)$, then $H\mapsto H\circ F$ and $H\mapsto F\circ H$ are endomorphisms of $\mathcal{H}_i$. We define the linear operator $\Gamma\colon \mathcal{H}_i\to \mathcal{H}_i$ as $H\mapsto A\circ H\circ A^{-1}.$
\end{definition}
The next lemma justifies the terms ``stable'' and ``unstable''.
\begin{lemma}\label{figo}
The stable subspace $\mathcal{S}_i$ is $\Gamma$-totally invariant and $\rho(\Gamma|_{\mathcal{S}_i})<1$. Indeed  $$\sp (\Gamma|_{\mathcal{S}_i})= \{\lambda_j\lambda^{-I}: X^j_I\in \mathfrak{B}_s\}.$$
The unstable subspace  $\mathcal{U}_i$ is $\Gamma$-totally invariant and  $\rho(\Gamma^{-1}|_{\mathcal{U}_i})<1.$
Indeed $$\sp (\Gamma|_{\mathcal{U}_i})= \{\lambda_j\lambda^{-I}: X^j_I\in \mathfrak{B}_u\}.$$ 
\end{lemma}
\begin{proof}

Since the $\Gamma$-invariance is an straightforward calculation, we prove the statement concerning the spectrum of $\Gamma|_{\mathcal{S}_i}.$ The automorphism $A$ is conjugate to any automorphism obtained multiplying the underdiagonal by a positive constant. Thus there exists a continuous path  $\gamma\colon [0,1]\to \mathcal{A}(\mathbb{C}^q)$ such that  $\gamma(0)=A$ and $\gamma(1)=(\lambda_1z_1,\dots,\lambda_q z_q )$, with $\gamma(0)$ conjugated to $\gamma(t)$ for all $t\in[0,1)$. 

Let $M\in \mathcal{A}(\mathbb{C}^q)$. Define $\Xi(M)\in \mathcal{A}(\mathcal{S}_i)$ as $H\mapsto M\circ H\circ M^{-1}.$ 
If $B=M\circ A\circ M^{-1}$, the linear operator $\Gamma|_{\mathcal{S}_i}=\Xi(A)$ is conjugate to the linear operator $\Xi(B).$ Indeed $$B\circ H\circ B^{-1}=M\circ A\circ M^{-1}\circ H\circ M\circ A^{-1}\circ M^{-1},$$ thus $\Xi(B)=     \Xi(M)\circ\Xi(A)\circ{\Xi(M)}^{-1}.$

We have $\lim_{t\to 1}\Xi(\gamma(t))= \Xi(\lambda_1z_1,\dots,\lambda_q z_q)$ and $\Gamma|_{\mathcal{S}_i}=\Xi(A)=\Xi(\gamma(0))$ is conjugate to $\Xi(\gamma(t))$  for all $t\in[0,1)$. Thus $$\sp(\Gamma|_{\mathcal{S}_i})=\sp(\Xi(\lambda_1z_1,\dots,\lambda_q z_q)).$$

It is easy to see that the linear operator $\Xi(\lambda_1z_1,\dots,\lambda_q z_q)$ is diagonalizable and that the basis $\mathfrak{B}_s$ is a basis of eigenvectors such that $$[\Xi(\lambda_1z_1,\dots,\lambda_q z_q)](X^j_I)= \lambda_j\lambda^{-I} X^j_I.$$ Thus $$\sp(\Gamma|_{\mathcal{S}_i})=\sp(\Xi(\lambda_1z_1,\dots,\lambda_q z_q))=   \{\lambda_j\lambda^{-I}: X^j_I\in \mathfrak{B}_s\}   .$$

The same argument works for the spectrum of  $\Gamma|_{\mathcal{U}_i}.$
\end{proof}

\begin{proposition}\label{core}
Let $D\subset \C^q$ be a complete hyperbolic domain.
Let $(\v_{n,m}\colon D\to D)$ be a dilation $\mathbb{N}$-evolution family such that $\v_{n,n+1}(z)=A (z)+O(|z|^2)$ with $A$ in optimal form.
Then for each $i\geq 2$   there exist
\begin{enumerate}
\item[i)] a family  $(k^i_n)$ of polynomial maps $k_n(z)=z+O(|z|^2)$ with uniformly bounded degree and uniformly bounded coefficients, and
\item[ii)] a triangular dilation evolution family $(T^i_{n,m})$  with $T^i_{n,n+1}(z)=A(z)+O(|z|^2)$,  $$\deg T^i_{n,n+1}\leq i-1,$$ and uniformly bounded coefficients such that for all $n\geq 0$,
\begin{equation}\label{rudin} k^i_{n+1}\circ \varphi_{n,n+1}-T^i_{n,n+1}\circ k^i_n=O(|z|^i).
\end{equation}
\end{enumerate}
 If no multiplicative real resonance occurs among the eigenvalues of $A$, then the family $(T^i_{n,m})$ is the linear family $(A^{m-n})$.
\end{proposition}
\begin{proof}
For $i=2$ set $T^2_{n+1,n}=A,\ k^2_{n}=\id,$ and we are done since $A$ is a triangular mapping.
Now assume that (\ref{rudin}) holds for $i\geq 2$. We can rewrite (\ref{rudin}) as 
\begin{equation}\label{rudin2}
 k^i_{n+1}\circ \varphi_{n,n+1}-T^i_{n,n+1}\circ k^i_n=P_{n,n+1}+O(|z|^{i+1}),
\end{equation}
where $(P_{n,n+1})$ is a bounded sequence in $\mathcal{H}_i.$
Define $R_{n,n+1}\doteq\pi_r(P_{n,n+1})$ which is in the real resonant subspace $\mathcal{R}_i$, and  $N_{n,n+1}\doteq P_{n,n+1}-R_{n,n+1}\in\mathcal{S}_i\oplus \mathcal{U}_i $.  Set $$T^{i+1}_{n,n+1}\doteq T^{i}_{n,n+1}+R_{n,n+1},$$ which is still a triangular dilation $\mathbb{N}$-evolution family with uniformly bounded degree and uniformly bounded coefficients since $R_{n,n+1}$ is a triangular mapping  thanks to (\ref{risontriang}), and set $$ k_n^{i+1}\doteq k_n^i+H_n\circ k_n^i,$$ where $(H_n)$ is an unknown bounded sequence in $\mathcal{H}_i$.  
\begin{align} k^{i+1}_{n+1}\circ \varphi_{n,n+1}&-T^{i+1}_{n,n+1}\circ k^{i+1}_n=\nonumber\\
&=(k_{n+1}^i+H_{n+1}\circ k_{n+1}^i)\circ \v_{n,n+1}-(T^{i}_{n,n+1}+R_{n,n+1})\circ(k_n^i+H_n\circ k_n^i)\nonumber \\
&=P_{n,n+1}-R_{n,n+1}+H_{n+1}\circ A -A \circ H_n +O(|z|^{i+1}) \nonumber\\
&=N_{n,n+1}+H_{n+1}\circ A -A \circ H_n +O(|z|^{i+1}).\nonumber
\end{align}
Thus to end the proof we need to prove the existence of a  bounded sequence $(H_n)$ of elements of  $\mathcal{H}_i$ which satisfies
\begin{equation}\label{nuova} 
N_{n,n+1}=A\circ H_n-H_{n+1}\circ A,
\end{equation}
that is  a bounded  solution $(H_n)$ of the  {\sl homological difference equation} $$H_{n+1}=A\circ H_n\circ A^{-1}-N_{n,n+1}\circ A^{-1}.$$ 

Define  $B_n\doteq -N_{n,n+1}\circ A^{-1}$. In the proof of Lemma \ref{figo} we proved that  $\mathcal{S}_i$ and  $\mathcal{U}_i$ are invariant by the linear operator $H\mapsto H\circ A^{-1}$, thus $B_n\in \mathcal{S}_i\oplus \mathcal{U}_i$. Define  $B_n^s\doteq \pi_s(B_n),\ B_n^u\doteq \pi_u(B_n).$
If $n\geq 1$ it is easy to prove by induction that
\begin{equation}\label{formulona}
H_n=\Gamma^n (H_0)+\sum_{j=0}^{n-1} \Gamma^j(  B_{n-1-j})=\Gamma^n (H_0)+\sum_{j=0}^{n-1} \Gamma^{n-1-j}(  B_j).
\end{equation}

We have
\begin{align}
H_n&=\Gamma^n (H_0)+\sum_{j=0}^{n-1}\Gamma^j(B^s_{n-1-j})+\sum_{j=0}^{n-1} \Gamma^j( B^u_{n-1-j})\nonumber\\
&=\sum_{j=0}^{n-1}\Gamma^j(B^s_{n-1-j})+ \Gamma^{n-1}(\Gamma(H_0)+\sum_{j=0}^{n-1}\Gamma^{-j}(B_j^u)).\nonumber
\end{align}

Recall that if $V$ is a complex vector space, and  $L\in \mathcal{L}(V)$, then the spectral radius of $L$ satisfies $\rho(L)=\inf_{ \|\cdot\|\in\mathcal{I}}\{\|L\|\},$ where $\mathcal{I}$ is the set of all operator norms induced by a norm on $V$. Hence by Lemma \ref{figo}  there exist a norm $\|\cdot\|_s$ on $\mathcal{S}_i$ and a norm $\|\cdot\|_u$ on $\mathcal{U}_i$ such that $\|\Gamma|_{\mathcal{S}_i}\|_s<1, \|\Gamma^{-1}|_{\mathcal{U}_i}\|_u<1.$ Define a norm on $\mathcal{S}_i\oplus \mathcal{U}_i$ by $$\|H\|\doteq\|\pi_s(H)\|_s+\ \|\pi_u(H)\|_u.$$

Since $(B_n^s)$ is bounded there exists $C>0$ such that $$\|\sum_{j=0}^{n-1}\Gamma^j(B^s_{n-1-j})\|\leq\sum_{j=0}^{\infty}\|\Gamma^j(B^s_{n-1-j})\|_s\leq C,\quad n\geq 0.$$ 

Since $(B_n^u)$ is bounded, $\sum_{j=0}^{\infty}\| \Gamma^{-j}( B_{j}^u)\|_u<+\infty,$ thus we can define $$H_0\doteq-\Gamma^{-1}(\sum_{j=0}^{\infty}\Gamma^{-j}( B_{j}^u))\in\mathcal{U}_i.$$ With this definition,
$$\|H_n\| \leq C+\|\Gamma^{n-1}(\sum_{j=n}^{\infty}\Gamma^{-j}( B_{j}^u))\|_u
=C+\|\sum_{j=1}^{\infty}\Gamma^{-j}( B_{n-1+j}^u)\|_u,$$
 and since $$\|\sum_{j=1}^{\infty}\Gamma^{-j}( B_{n-1+j}^u)\|_u\leq \sum_{j=1}^{\infty}\|\Gamma^{-j}( B_{n-1+j}^u)\|_u\leq C',$$ we have $\|H_n\|\leq C+C'$.

If no multiplicative real resonance occurs among the eigenvalues of $A$, then  $$\mathcal{R}_i=\varnothing,\quad \mbox{for all}\ i\geq 2,$$ and thus  $T^{i+1}_{n,m}=T^{i}_{n,m}$ for all $ i\geq 2$, which gives $$T^{i}_{n,m}=T^{2}_{n,m}=A^{m-n},\quad \mbox{for all}\  i\geq 2.$$

\end{proof}
\begin{remark}
Let $p\geq 0$ be the smallest integer such that $|\lambda_1^p|<|\lambda_q|.$ Then if $i\geq p$  we have $\pi_r(P_{n,n+1})=0$ in $\mathcal{H}_i$. Hence $T^i_{n,n+1}=T^p_{n,n+1}$ for any $i\geq p$.
\end{remark}
\begin{proposition}\label{discreteresult}
Let $D\subset \C^q$ be a complete hyperbolic domain.
Let $(\v_{n,m},\colon D\to D)$ be a dilation  $\mathbb{N}$-evolution family such that $\v_{n,n+1}(z)=A(z)+O(|z|^2)$ with $A$ in optimal form.
Then there exists a triangular dilation $\mathbb{N}$-evolution family $(T_{n,m})$ with bounded degree and bounded coefficients locally conjugate to $(\v_{n,m})$. If no multiplicative real resonance occurs among the eigenvalues of $A$, then $(\v_{n,m})$ is locally conjugate to its linear part $(A^{m-n})$.
\end{proposition}
\begin{proof}
Let $\alpha$ be such that  $ \max_{z\in \mathbb{C}^q}\frac{|A(z)|}{|z|}<\alpha<1.$
Let $(T_{n,m}^i)$ and $(k_n^i)$ be the families given by Proposition \ref{core}.
Let $p\geq 0$ be as in previous remark. Define $(T_{n,m})\doteq(T_{n,m}^p)$.
Let $\beta>0$ be the constant given by Lemma \ref{peters} for $(T_{n,m})$. Let $\ell\geq 0$ be an integer such that  $\alpha^\ell<1/\beta,$ and define $(k_n)\doteq (k_n^\ell).$
By Proposition \ref{core},  $$k_{n+1}\circ \varphi_{n,n+1}-T_{n,n+1}\circ k_n=O(|z|^\ell),$$ thus $$T_{n,n+1}^{-1}\circ k_{n+1}\circ \varphi_{n,n+1}- k_n=O(|z|^\ell).$$ 

By Lemma \ref{estimate}   there exists $r>0$ (we can assume $0<r<1/2$) such that on $r\mathbb{B}$ we have $|\varphi_{n,n+1}(z)|\leq \al |z|$ and $|T_{n,n+1}(z)|\leq \al |z|$ for all $n\geq 0$. Thus for $\zeta\in r\mathbb{B}$ we have $$|\varphi_{0,m}(\zeta)|<r\al^{m}.$$ Thanks to Lemma \ref{taylor}  there exists $C>0$ such that on $r\mathbb{B}$, 
$$|T_{m,m+1}^{-1} \circ k_{m+1}\circ \varphi_{m,m+1}(\zeta)- k_m(\zeta)|\leq C|\zeta|^\ell,\quad m\geq 0.$$ Hence
$$|T_{m,m+1}^{-1}\circ k_{m+1}\circ\v_{0,m+1}(\zeta)-k_m\circ\v_{0,m}(\zeta)|\leq C|\v_{0,m}(\zeta)|^\ell\leq C r^\ell\al^{\ell m}.$$

Let $\Delta$ be the unit polydisc. There exists $s\mathbb{B}\subset r\mathbb{B}$ such that $$T_{m,m+1}^{-1}\circ k_{m+1}\circ\v_{0,m+1}(s\mathbb{B})\subset \frac{1}{2}\Delta$$ and $$k_m\circ\v_{0,m}(s\mathbb{B})\subset \frac{1}{2}\Delta.$$ Indeed the families $(k_m)$ and $(T_{m,m+1}^{-1})$ are uniformly bounded  on $r\mathbb{B}$ and thus equicontinuous in $0$.

 Hence  Lemma \ref{peters}  applies to get on $s\mathbb{B}$,
$$|T_{0,m+1}^{-1}\circ k_{m+1}\circ\v_{0,m+1}(\zeta)-T_{0,m}^{-1}\circ k_m\circ\v_{0,m}(\zeta)|\leq C r^\ell(\beta\al^\ell)^m.$$
Likewise it is easy to see that for all $m\geq n\geq 0$,
$$|T_{n,m+1}^{-1}\circ k_{m+1}\circ\v_{n,m+1}(\zeta)-T_{n,m}^{-1}\circ k_m\circ\v_{n,m}(\zeta)|\leq C r^\ell(\beta\al^\ell)^{m-n}.$$

Since $\alpha^\ell<1/\beta$ for all $n \geq 0$ there exists a holomorphic mapping $h_n$ on $s\mathbb{B}$ such that $$h_n=\lim_{m\to\infty} T_{n,m}^{-1}\circ k_m\circ \varphi_{n,m}$$  uniformly on compacta.  Each $h_n$ is bounded by $|k_n|+\sum_{j=0}^\infty C r^\ell(\beta\al^\ell)^j,$ hence they are uniformly bounded. Moreover $$h_m\circ\varphi_{n,m}=\lim_{j\rightarrow \infty} T_{m,j}^{-1}\circ k_j\circ\varphi_{m,j}\circ\varphi_{n,m}=\lim_{j\rightarrow \infty}  T_{n,m}\circ T_{j,n}\circ k_j\circ\varphi_{n,j}= T_{n,m}\circ h_n.$$

If no  multiplicative real resonance occurs among the eigenvalues of $A$, then $(T_{n,m})=(A^{m-n})$.
\end{proof}

We can now prove an existence result for $\mathbb{N}$-Loewner chains.
\begin{proposition}\label{oddio}
Let $D\subset \C^q$ be a complete hyperbolic domain.
Let $(\v_{n,m}\colon D\to D)$ be a dilation $\mathbb{N}$-evolution family, $\v_{n,n+1}(z)=A(z)+O(|z|^2)$. Then there exists a $\mathbb{N}$-Loewner chain $(f_n\colon D\to \mathbb{C}^q)$ with $\cup_{n\geq 0}f_n(D)=\C^q$ associated with $(\v_{n,m})$, which is locally bounded if no multiplicative real resonance occurs among the eigenvalues of $A$.
\end{proposition}
\begin{proof}
Up to a linear change of coordinates, we can assume that $A$ is in optimal form. By Proposition \ref{discreteresult}  there exists a triangular dilation $\mathbb{N}$-evolution family $(T_{n,m})$ with bounded degree and bounded coefficients locally conjugate to $(\v_{n,m})$. Let $(h_n)$ be the  family of intertwining mappings  given by Propostion \ref{inter}. Then $(T_{0,n}^{-1}\circ h_n)$ is a $\mathbb{N}$-Loewner chain $(f_n)$  associated with $(\v_{n,m})$. Indeed  $$T_{0,m}^{-1}\circ h_m\circ \v_{n,m}(z)= T_{0,m}^{-1}\circ T_{n,m}\circ h_n(z)=T_{0,n}^{-1}\circ h_n(z),\quad z\in D,0\leq n\leq m.$$ 

By Lemma \ref{koebe} there exists a ball $s\mathbb{B}\subset\bigcap_{n\geq 0} h_n(D)$. Hence $$\bigcup_{n\geq 0}T_{0,n}^{-1}(h_n(D))\supset \bigcup_{n\geq 0}T_{0,n}^{-1}(s\mathbb{B})=\mathbb{C}^q$$ by Lemma \ref{peters}.

If no  multiplicative real resonance occurs among the eigenvalues of $A$, then $(T_{n,m})=(A^{m-n})$, and the chain $(A^{-n}\circ h_n)$ is  locally bounded.
\end{proof}

Now we can go back to the continuous-time setting.
\begin{definition}
A $\mathbb{R}^+$-evolution family $(\v_{s,t})$ on $D$   is a {\sl dilation $\mathbb{R}^+$-evolution family} if for all $ 0\leq s\leq t$,
\begin{equation}\label{continuous dilation type}
 \varphi_{s,t}(z)=e^{\Lambda(t-s)}(z)+O(|z|^2),
\end{equation}
where the eigenvalues of $\Lambda\in \mathcal{L}(\mathbb{C}^q)$ have strictly negative real part.

A $\mathbb{R}^+$-Loewner chain $(f_t\colon D\to \mathbb{C}^q)$ is a {\sl locally bounded $\mathbb{R}^+$-Loewner chain} if for all $t\geq 0$, $$f_t(z)=e^{-\Lambda t}(z)+O(|z|^2),$$ where the eigenvalues of $\Lambda\in \mathcal{L}(\mathbb{C}^q)$ have strictly negative real part and  the family $(e^{\Lambda t}\circ f_t)$ is  uniformly bounded in a neighborhood of the origin.

If we restrict time to integer values in a dilation $\mathbb{R}^+$-evolution family $(\varphi_{s,t})$  we obtain its {\sl discretized} dilation $\mathbb{N}$-evolution family $(\varphi_{n,m})$. We have $$\v_{n,n+1}(z)=e^{\Lambda }(z)+O(|z|^2).$$ 

An {\sl additive real resonance} is an identity $$\Re(\sum_{j=1}^N k_j\alpha_j)=\Re\alpha_l,$$ where $k_j\geq0$ and $\sum_j k_j\geq 2$. 
Recall that  $\alpha $ is an eigenvalue of $\Lambda$ with algebraic molteplicity $m$ if and only if $e^\alpha$ is an eigenvalue of $e^\Lambda$ with algebraic molteplicity $m$. 
Hence additive real resonances of $\Lambda$ correspond to multiplicative real resonances of $e^\Lambda.$
\end{definition}
\begin{lemma}\label{discretetocontinuous}
Let $D$ be a complete hyperbolic domain.
Let $(\varphi_{s,t}\colon D\to D)$ be a dilation $\mathbb{R}^+$-evolution family, and let $(\varphi_{n,m}\colon D\to D)$ be its discretized evolution family. Assume there exists a  $\mathbb{N}$-Loewner chain $(f_n)$ associated with  $(\varphi_{n,m})$. Then we can extend it in a unique way to a  $\mathbb{R}^+$-Loewner chain associated with $(\varphi_{s,t})$. If $(f_n)$ is a locally bounded $\mathbb{N}$-Loewner chain, then also $(f_s)$ is locally bounded.
\end{lemma}
\begin{proof}
For all $s\in\mathbb{R}^+$ define $f_s=f_j\circ \varphi_{s,j},$ where $j$ is an integer such that $s\leq j$. The family $(f_s)$ is a  $\mathbb{R}^+$-Loewner chain associated with $(\varphi_{s,t})$ (cf. \cite[Lemma 8.5]{Arosio}). 
 If $(f_n)$ is a locally bounded $\mathbb{N}$-Loewner chain, then there exists  $r>0$  such that  the family  $(e^{\Lambda n}\circ f_n)$  is uniformly bounded on the Kobayashi ball  $\Omega(0,r)$  centered in the origin of radius $r>0$. 
For each $s\geq 0$ define $m_s$ as the smallest integer greater than $s$. One has $$ e^{\Lambda s}\circ f_s=e^{\Lambda s}\circ f_{m_s}\circ\v_{s,m_s}= e^{\Lambda (s-m_s)}\circ e^{\Lambda m_s}\circ f_{m_s} \circ\v_{s,m_s},$$
and since $\v_{s,m_s}(\Omega(0,r))\subset\Omega(0,r)$ and  $m_s-s\leq 1$,  the family $ (e^{\Lambda s}\circ f_s)$   is uniformly bounded on  $\Omega(0,r)$.
\end{proof}

\begin{proposition}\label{continuousresult}
Let $D\subset \C^q$ be a complete hyperbolic domain.
Let $(\varphi_{s,t}\colon D\to D)$ be a dilation $\mathbb{R}^+$-evolution family, $\v_{s,t}(z)=e^{\Lambda(t-s)}(z)=O(|z|^2)$. Then there exists a  $\mathbb{R}^+$-Loewner chain $(f_s\colon D\to
\mathbb{C}^q)$  with $\cup_{t\geq 0}f_t(D)=\C^q$ associated with $(\varphi_{s,t})$, which is locally bounded if no additive real resonance occurs among the eigenvalues of $\Lambda$.
\end{proposition}
\begin{proof}
Let  $(\varphi_{n,m}\colon D\to D)$ be the discretized evolution family of $(\v_{s,t}\colon D\to D)$.  Since no additive real resonance occurs in $\Lambda$, no multiplicative real resonance occurs in $A=e^{\Lambda}.$
The result follows from Proposition \ref{oddio} and Lemma \ref{discretetocontinuous}. 
\end{proof}
\begin{remark}
If no additive real resonance occurs among the eigenvalues of $\Lambda$ then by the proof of Proposition \ref{continuousresult} there exists a family $(h_n)$ of tangent to identity polynomial mappings of uniformly bounded degree and uniformly bounded coefficients  such that $$f_s= \lim_{m\to\infty} e^{-\Lambda m}\circ h_m\circ \varphi_{s,m}.$$
\end{remark}

\section{The Loewner PDE}\label{PDE} 
\begin{definition}
Let $D$ be a domain in $\mathbb{C}^q$ containing $0$ and let $d\in [1,+\infty]$.
A \textit{dilation Herglotz vector field of order $d\geq 1$} on $D$ is a mapping
$$G\colon D\times \mathbb{R}^+\to \mathbb{C}^q$$
satisfying
\begin{itemize}
\item[a)] for all $z\in D$ the map $t\mapsto H(z,t)$ is measurable,
\item[b)] for a.e.  $t\geq 0$ the map $z\mapsto H(z,t)$ is an infinitesimal generator on $D$ of the form $$H(z,t)=\Lambda(z)+O(|z|^2)$$ where the eigenvalues of $\Lambda\in \mathcal{L}(\mathbb{C}^q)$ have strictly negative real part.
\item[c)] for any compact set $K\subset D$ and  there exists a  function $c_{K}\in L^d_{loc}(\R^+,\R^+)$ such that $$|H(z,t)|\leq c_{K}(t),\quad z\in K, t\geq 0.$$
\end{itemize}
The partial differential equation 
\begin{equation}\label{LoewnerPDE}
\frac{\partial f_t(z)}{\partial t}=-df_t(z)H(z,t)\quad a.e.\ t\geq 0,\ z\in D, 
\end{equation}
 where $H(z,t)$ is a dilation  Herglotz vector field, is called the {\sl Loewner PDE}.
\end{definition}

The following is our main result.
\begin{theorem}\label{PDEsolution}
Let $D\subset \C^q$ be a complete hyperbolic domain and let $H(z,t)=\Lambda(z)+O(|z|^2)$ be a dilation Herglotz vector field on $D$ of order $d\geq 1$.
Then the Loewner PDE (\ref{LoewnerPDE}) admits a solution given by a family of univalent mappings $(f_t\colon D\to \mathbb{C}^q)$, such that $\cup_{t\geq 0}f_t(D)=\mathbb{C}^q$, and which is locally absolutely continuous of order $d$ in the following sense:
 for any compact set $K\subset D$  there exists a function  $k_{K}\in L^d_{loc}(\R^+,\mathbb{R}^+)$ such that
 \begin{equation}\label{LCdef}
|f_s(z)- f_t(z)|\leq \int_s^tk_{K}(\xi)d\xi,\quad z\in K,0\leq s\leq t.
 \end{equation}
 If no additive real resonance occurs among the eigenvalues of $\Lambda$  then 
the family $(e^{\Lambda t}\circ f_t)$ is uniformly bounded in a neighborhood of the origin. Any locally absolutely continuous solution  given by a family of holomorphic mappings $(g_t\colon D\to \mathbb{C}^q)$ is of the form $(\Psi\circ f_t)$, where $\Psi\colon \mathbb{C}^q\to \mathbb{C}^{q}$ is holomorphic. 
\end{theorem}

\begin{proof}
Since by assumption $D$ is complete hyperbolic, \cite{Arosio-Bracci} yields that the solution of the Loewner ODE
\begin{equation}
\begin{cases}
\frac{\de}{\de t}\v_{s,t}(z)= H(\v_{s,t}(z), t), \quad \hbox{a.e. } t\in[s,\infty),\\
\v_{s,s}(z)=z,\quad s\geq 0.
\end{cases}
\end{equation}
is a $\mathbb{R}^+$-evolution family $(\v_{s,t}(z)=e^{\Lambda (t-s)}(z)+O(|z|^2))$ which is locally absolutely continuous of order $d$ in the following sense:
 for any compact set
  $K\subset D$ there exists a  function $C_{K}\in
  L^d_{loc}(\R^+,\R^+)$ such that
  \begin{equation}\label{ck-evd}
|\v_{s,t}(z)- \v_{s,u}(z)|\leq \int_{u}^t
C_{K}(\xi)d \xi, \quad z\in K,\  0\leq s\leq u\leq t.
  \end{equation}
By Proposition \ref{continuousresult} there exists a  $\mathbb{R}^+$-Loewner chain $(f_t\colon D\to \C^q)$  with $\cup_{t\geq 0}f_t(D)=\C^q$ associated with $(\v_{s,t})$.
By \cite[Theorem 4.10]{Arosio-Bracci-Hamada-Kohr} the chain $(f_t)$ is of locally absolutely continuous of order $d$, and by \cite[Theorem 5.2]{Arosio-Bracci-Hamada-Kohr} it solves the Loewner PDE (\ref{LoewnerPDE}). If  no additive real resonance occurs among the eigenvalues of $\Lambda$ then by Proposition  \ref{continuousresult} the chain $(f_t)$ is locally bounded. Any locally continuous solution of  (\ref{LoewnerPDE}) is  by \cite[Theorem 5.2]{Arosio-Bracci-Hamada-Kohr} a $\mathbb{R}^+$-Loewner chain associated with  $(\v_{s,t})$, thus by Theorem \ref{thuniqueness} it is of the form  $(\Psi\circ f_t)$, where $\Psi\colon \mathbb{C}^q\to \mathbb{C}^{q}$ is holomorphic. 
\end{proof}

\section{An univalence criterion}\label{univalencecriterionsection}
In this section we  generalize a classical univalence criterion in the unit disk due to Pommerenke.
We first need to generalize the notion of locally bounded $\mathbb{R}^+$-Loewner chains in order to include families of non-necessarily univalent mappings.

\begin{definition}\label{normalsub}
Let $D$ be a domain in $\mathbb{C}^q$ containing $0$.
A family $(f_t\colon D\to \mathbb{C}^q)$ is a {\sl locally bounded $\mathbb{R}^+$-subordination chain} if
\begin{itemize}
\item[i)]  for all $0\leq s\leq t$ there exists a holomorphic mapping  $\v_{s,t}\in \Hol(D,D)$   fixing $0$ and satisfying  $f_s=f_t\circ \v_{s,t}$, called {\sl transition mapping},
\item[ii))]  $f_t(z)=e^{-\Lambda t}(z)+O(|z|^2)$ for all $ t\geq 0$,  the eigenvalues of $\Lambda\in\mathcal{L}(\mathbb{C}^q)$ have strictly negative real part, and the family $(e^{\Lambda t}\circ f_t)$ is uniformly bounded in a neighborhood of the origin.
\end{itemize}

Now we can state Pommerenke's criterion.
\begin{theorem}[{\cite[Folgerung 6]{Pommerenke}}]
If $(f_t\colon\mathbb{D}\to\mathbb{C})$ is a locally bounded $\mathbb{R}^+$-subordination chain, then for all $t\geq 0$ the mapping $f_t$ is univalent.
\end{theorem}
This criterion has been generalized to the unit ball $\mathbb{B}\subset \C^q$, with different hypotheses,  by Pfaltzgraff \cite[Theorem 2.3]{Pfaltzgraff}, and by Graham and Kohr \cite[Theorem 8.1.6]{Graham-Kohr}. To generalize Pommerenke's criterion we do not assume the subordination chain to solve a Loewner PDE: we  only assume continuity.
A  locally bounded $\mathbb{R}^+$-subordination chain $(f_t\colon D\to\mathbb{C}^q)$ is {\sl continuous} if  the mapping $t\mapsto f_t$ is continuous with respect to the topology of uniform convergence on compacta on $\Hol(D,\mathbb{C}^q)$.
\end{definition}
\begin{proposition}\label{univalencecriterion}
Let $D\subset \C^q$ be a complete hyperbolic domain. If $(f_t\colon D\to \C^q)$ is a continuous locally bounded $\mathbb{R}^+$-subordination chain, then for all $t\geq 0$ the mapping $f_t$ is univalent.
\end{proposition}
\begin{proof}
We have to show that $f_t$ is univalent for all $t\geq 0$.
The identity principle yields that for all $0\leq s \leq t$ there exists a unique transition mapping $\v_{s,t}\in \Hol(D,D)$ satisfying  $f_s=f_t\circ \v_{s,t}$ and  fixing $0$, since $f_t$ is locally invertible at $0$. Thus the family $(\v_{s,t})_{0\leq s \leq t}$  satisfies  $\v_{s,s}=\id$ for all $s\geq 0$. Moreover one has $$\v_{u,t}\circ\v_{s,u}=\v_{s,t},\quad 0\leq s\leq u\leq t.$$ 
Indeed for all $0\leq s\leq u\leq t$, $$f_t\circ \v_{u,t}\circ \v_{s,u}=f_u\circ\v_{s,u}= f_s=f_t\circ \v_{s,t},$$ and the assertion follws since the transition mapping is unique. Moreover by the chain rule $df_s(0)=d f_t(0)\circ d\v_{s,t}(0)$, hence $d\v_{s,t}(0)=e^{\Lambda(t-s)}$.

We claim that $\lim_{t\to s+}\v_{s,t}=\id$ for all $s\geq 0$. Indeed since $(\v_{s,t}\colon D\to D)_{0\leq  s\leq t}$ is a normal family and $\v_{s,t}(0)=0$ for all $0\leq s\leq t$, any  sequence $(\v_{s,t_n})$ with $t_n\to s$ admits a subsequence $(\v_{s,t_{n_k}})$ converging on compacta to a mapping $\v\in \Hol(D,D)$, and by $f_s=f_{t_{n_k}} \circ\v_{s,t_{n_k}}$ we obtain $f_s=f_s\circ \v$, thus $\v=\id$. This proves that $\lim_{t\to s+}\v_{s,t}=\id$. In the same way, $\lim_{s\to t-}\v_{s,t}=\id.$ 

This implies that $\v_{s,t}$ is univalent for all $0\leq s\leq t$. Indeed suppose  there exists $0<s<t$ and $z\neq w$ contained in a Kobayashi ball $\Omega(0,\ell)$  such that $\v_{s,t}(z)=\v_{s,t}(w).$ Set $r\doteq \inf \{u\in [s,t]: \v_{s,u}(z)=\v_{s,u}(w)\}.$ 
Since $\lim_{u\to s+}\v_{s,u}=\id$ uniformly on compacta, we have $r>s$. If $u\in (s,r)$, $$\v_{u,r}(\v_{s,u}(z))=\v_{u,r}(\v_{s,u}(w)),$$ and since $\v_{s,u}(z)\neq \v_{s,u}(w),$ the mapping $\v_{u,r}$ is not univalent on $$\bigcup_{u\in(s,r)}\v_{s,u}(z)\cup \v_{s,u}(w)\subset \Omega(0,\ell).$$ Since  $D$ is complete hyperbolic, by \cite[Proposition 1.1.9]{Kobayashi} one has  $\Omega(0,\ell)\subset\subset D$.  But $\lim_{u\to r-} \v_{u,r}=\id$ uniformly on compacta which is a contradiction since the identity mapping is univalent.

Define $h_s\doteq e^{\Lambda s}\circ f_s$, for all $ s\geq 0$. By hypothesis the family $(h_s\colon D\to \C^q)$ is uniformly bounded in a neighborhood of the origin. From $f_t\circ\v_{s,t}=f_s$ we obtain $$h_t\circ \v_{s,t}=e^{\Lambda(t-s)}\circ h_s.$$ Fix $s\geq 0$. The dilation $\mathbb{N}$-evolution family $(\v_{s+n,s+m})_{0\leq n\leq m}$  is locally conjugate to $(e^{\Lambda(m-n+s)})_{0\leq n\leq m}$ by means of the intertwining mappings $(h_{s+n})_{n\geq 0}$. By Proposition \ref{interuniv} the mapping $h_s$ is univalent, thus $f_s=e^{-\Lambda s}\circ h_s$ is univalent.
\end{proof}
As a corollary one easily obtains the following.
\begin{corollary}
Let $D\subset \mathbb{C}^q$ be a complete hyperbolic domain and let $H(z,t)=\Lambda(z)+O(|z|^2)$ be a dilation Herglotz vector field on D.  
Let  $(f_t\colon D\to \C^q)$ be a locally absolutely continuous family of holomorphic mappings which solves the Loewner PDE (\ref{LoewnerPDE}) and assume that the family $(e^{\Lambda t}\circ f_t)$ is uniformly bounded in a neighborhood of the origin. Then for all $t\geq 0$ the mapping $f_t$ is univalent.
\end{corollary}

\appendix

\section{Ausiliary Lemmas}
For the convenience of the reader, we recall here some auxiliary Lemmas, in the form used in the proofs.
\begin{lemma}{\cite[Lemma 2.2]{Arosio}}\label{taylor}
Let $A\in \mathcal{L}(\mathbb{C}^q)$.
Let $\mathcal{F}$ be a family of holomorphic mappings $(f\colon r\mathbb{B}\rightarrow \mathbb{C}^q)$, bounded by $M>0$, and let $k\geq 2 $ such that  $f(z)-A(z)=O(|z|^k)$ for all $f\in \mathcal{F}$.
Then there exists $C_k>0$ such that   $|f(z)-A(z)|\leq C_k|z|^k$ for all $z\in r\mathbb{B}$. 
\end{lemma}
\begin{lemma}{\cite[Lemma 2.3]{Arosio}}\label{estimate}
Let $A\in \mathcal{L}(\mathbb{C}^q)$, and let $D$ be a domain containing $0$.
Let $\mathcal{F}$ be a family of holomorphic mappings $(f\colon D\rightarrow \mathbb{C}^q)$, bounded by $M>0$, and satisfying $f(z)=A(z)+O(|z|^2)$. Let $\al>0$ be such that  $ \max_{z\in \mathbb{C}^q}\frac{|A(z)|}{|z|}<\alpha$.
Then there exists  $s>0$ such that if $f\in \mathcal{F}$ then  $|f(z)|\leq \alpha |z|$ for all $ |z|\leq s.$
\end{lemma}
\begin{lemma}{\cite[Lemma 2.5]{Arosio}}\label{koebe}
Let $A\in \mathcal{A}(\mathbb{C}^q)$, and let $D$ be a domain containing $0$.
Let $\mathcal{F}$ be a family of holomorphic mappings $(f\colon D\rightarrow \mathbb{C}^q)$, bounded by $M>0$, and satisfying $f(z)=A(z)+O(|z|^2)$.
There exist $r>0$ and $s>0$ such that if $f\in \mathcal{F}$ then $f$ is univalent on $r\mathbb{B}$, and $s\mathbb{B}\subset f(r\mathbb{B})$.
\end{lemma}

The following Lemma  is stated in \cite[Lemma 11]{Peters} as a simple generalization of   \cite[Lemma 1, Appendix]{Rudin-Rosay}. A proof can be found in \cite[Corollary 4.4, Lemma 4.5]{Arosio}.
\begin{lemma}{\cite[Corollary 4.4]{Arosio}}\label{peters}
Let $\Delta\subset \C^q$ be the unit polydisc.

Let $(T_{n,m})$ be a triangular dilation $\mathbb{N}$-evolution family of uniformly bounded degree and uniformly bounded coefficients. Then
\begin{itemize}
\item[a)]
 there exists $\beta\geq 0$ such that for all $k\geq 0$, $$|T_{0,k}^{-1}(z)-T_{0,k}^{-1}(z')|\leq \beta^k|z-z'|,\quad z,z'\in \frac{1}{2}\Delta.$$

\item[b)] $T_{0,n}(z)\rightarrow 0$ uniformly on compacta and  for each neighborhood $V$ of $0$ we have $\bigcup_{n=1}^\infty T_{0,n}^{-1}(V)=\mathbb{C}^q.$
\end{itemize}
\end{lemma}

\end{document}